\documentclass[a4paper,12pt]{article}
\usepackage[hmargin=2.5cm,vmargin=1.5cm]{geometry}
\usepackage{amsfonts}
\usepackage{amssymb}
\usepackage{amsmath}
\usepackage{amsthm}
\usepackage{amscd}
\usepackage{geometry}
\usepackage{array}
\usepackage{authblk}
\usepackage{float}

\usepackage{pst-node}
\usepackage{tikz-cd}
\usepackage{tikz-qtree}

\usepackage{mathtools}
\DeclarePairedDelimiter\ceil{\lceil}{\rceil}

\newtheorem{definition}{Definition}[section]
\newtheorem{theorem}{Theorem}[section]
\newtheorem{lemma}{Lemma}[section]
\newtheorem{proposition}{Proposition}[section]
\newtheorem{remark}{Remark}[section]
\newtheorem{corollary}{Corollary}[section]
\newtheorem{example}{Example}[section]

\newtheorem{assumption}{Assumption}[section]

\newcommand{\om}{\textnormal{\textbf{o}}}

\newcommand{\opsi}{\textnormal{\textbf{$\psi$}}}

\newcommand{\som}[2]{\overset{#2}{\underset{i=#1} {\textnormal{\large\textbf{O}}}    }}

\newcommand{\spsi}[2]{\overset{#2}{\underset{i=#1} {\textnormal{\Large\textbf{$\psi$}}}    }}

\begin{document}

\title{A unified analytic approach to O-metric inequalities and applications to fixed point theory}

\date{}
\author[1]{\small Olaoluwa, Hallowed O.}
\author[2]{\small Ige, Aminat O.}
\author[3]{\small Olaleru, Johnson O.}

\affil[1,3]{\small Department of Mathematics, University of Lagos, Akoka, Nigeria}
\affil[2]{\small Department of Mathematics, Lagos State University, Nigeria}

\affil[1]{\small email address:  holaoluwa@unilag.edu.ng}
\affil[2]{\small email address: aminat.ige@lasu.edu.ng (corresponding author)}
\affil[3]{\small email address: jolaleru@unilag.edu.ng}

\maketitle

\begin{abstract}
\noindent
The concept of O-metrics was recently introduced as a generalization of several metric-type structures by replacing the addition operation of the standard triangle inequality with a binary operation that may fail to be associative. This non-associativity naturally to generalized polygon inequalities and patterned compositions. In this work we develop an analytic framework for inequalities arising in such settings. By introducing generalized $\om$-series governed by admissible control functions $\varphi$, we establish a separation principle that resolves inequalities of the form $u \leq v \, \om \, \varphi(k,u)$. This result provides convergence criteria for patterned compositions and determines intervals for the admissible contraction parameter $k$. As application, some fixed point theorems of 
\'{C}iri\'{c} type are obtained for mappings on O-metric spaces, as well as corresponding results for b-metric spaces. The approach provides a unified analytic framework for studying contractive conditions and iterative processes in generalized metric spaces.
\vspace{5mm}
	
\\

	
\end{abstract}

\section{Introduction and Preliminaries}

Non‑associativity presents a non-trivial obstacle in extending metric convergence and fixed point theory beyond classical settings. In O-metric spaces, the ``distance" between points is composed through a binary operation $\om$ that need not be associative, forcing convergence arguments to rely on families of polygon inequalities rather than sums. While fixed point theory has motivated much of the recent literature (see \cite{rhoades}, \cite{Bakhtin1989}, \cite{Czerwik1993}, \cite{openball}, \cite{suzuki}, \cite{igez} and other works of the authors in \cite{Olaleru2009}, \cite{multiplito}, \cite{Olaoluwa2015}, \cite{inclo}), the decisive difficulty lies in controlling iterated non‑associative compositions.
\\

\noindent
In this paper we introduce a calculus of generalized $\om$-series indexed by patterns that distinguish non‑associative compositions, then prove a general principle for the separation of variables in some inequalities governed by control functions, without the use of classical algebraic inverses. The results on generalized series are independent of any self‑map and apply broadly to iterative schemes in O-metric-type structures. Fixed point theorems then appear naturally as applications.
\\


\noindent
We recall the definition of O-metrics (see \cite{bookchap}-\cite{ourpaper2}):  given a non-negative binary operation $\om$ on pairs of elements of an interval $I_a$ of non-negative real numbers containing some non-negative real  $a \in [0,\infty)$, an $\om$-metric on a non-empty set $X$ is a function  $d_\om: X \times X \to I_a$ such that
\begin{itemize}
\item[$(O_1)$] $d_\om(x,y)=a$ if and only if $x=y$;
\item[$(O_2)$] $d_\om(x,y)=d_\om(y,x)$;
\item[$(O_3)$] $d_\om(x,z) \leq d_\om(x,y) \, \om \, d_\om(y,z)$.
\end{itemize}
The condition $(O_3)$, called the triangle $\om$-inequality, together with condition $(O_1)$, naturally extend the notion of metrics in the classical sense and impose the super-idempotence $a \le a \, \om \, a$ at the self-similarity metric value $a$. The terminologies $a$-upward (respectively, $a$-downward) $\om$-metric are employed when $d_\om(x,y) \ge a$ (respectively, $d_\om(x,y) \le a$, in which case, it is assumed that $a>0$). The class of O-metrics contains all possible $\om$-metrics.

\subsection{Polygon Inequalities}

In the literature of fixed point theory,  binary operations often used for metric-types include addition for metrics, the maximum operation for ultra-metric spaces, scaled addition $u \om v = s(u+v)$ for some constant $s \ge 1$ for b-metrics, multiplication for multiplicative metrics, etc. It is worthy to note that, from a fixed point theory perspective, for any binary operation
\begin{equation}\label{obvious}
u \,  \om  \, v=\varphi(\varphi^{-1}(u)+\varphi^{-1}(v)),
\end{equation}
where $\varphi$ is a bijection $\varphi: [0, \infty) \to [a,\infty)$ such that $\varphi(0)=a$, 
fixed point theorems on metric spaces are transposable to $\om$-metric spaces as the $\om$-metric $d_\om$ is such that $\varphi^{-1} \circ d_\om$ is a metric and vice-versa. In fact, the binary operations satisfying (\ref{obvious}) are all associative, and the triangle $\om$-inequality is easily extended to a polygonal inequality
\begin{equation}
d_\om(x,y) \le \varphi \left(\sum_{i=0}^n \varphi^{-1}(d_\om (x_{i},x_{i+1})) \right)
\end{equation} 
for points $x=x_0, x_1,x_2,\ldots, x_{n+1}=y$. 
\\

\noindent
For a monotone nondecreasing binary operation $\om$ is not necessarily associative  (i.e. $u \, \om \, (v \, \om \, w)) \neq (u \, \om \,v) \, \om \, w$ in general), many inequalities are possible when applying the triangle $\om$-inequality on at least three points. Indeed, if $x_0,x_1,x_2,x_3$ are points in the O-metric space $(X,d_\om,a)$, the following inequalities hold: 
$$
\begin{array}{lll}
d_\om(x_0,x_3) &\leq & d_\om(x_0,x_1) \, \om \left(d(x_1,x_2) \, \om \,d_\om(x_2,x_3)\right)
\\
d_\om(x_0,x_3) &\leq & \left(d_\om(x_0,x_1) \, \om \, d_\om(x_1,x_2)\right) \, \om \, d_\om(x_2,x_3).
\end{array}$$
In fact, for points $x_0,x_1,x_2,\ldots,x_{n+1}$, with $n \geq 1$, there are  at most $C_n$ possible inequalities, where $C_{n}:=\frac{1}{n+1} {{2n}\choose{n}}$ is a Catalan number (see \cite{assoc2}). We write any of the following polygon inequalities
\begin{equation}
d_\om(x_0,x_{n+1}) \leq h\left(d_\om(x_0,x_1),d_\om(x_1,x_2),\ldots,d_\om(x_{n},x_{n+1})\right)   ~ \forall h \in \Omega_n, \label{comp_1}
\end{equation}
or simply,
\begin{equation}
d_\om(x_0,x_{n+1})  \leq d_\om(x_0,x_1) \, \om \, d_\om(x_1,x_2) \, \om\,  \cdots \, \om \,  d_\om(x_{n},x_{n+1}),\label{comp_2}
\end{equation}
where for non-negative real numbers $t_0,t_1,\ldots,t_n$, the expression $t_0 \, \om \,t_1 \, \om \, \cdots \, \om \,t_n$ denotes one of the $C_n$ possibilities of composing successively the terms $t_i$ starting from $t_0$ by the binary operation $\om$, and $\Omega_n$ denotes the set (of order at most $C_n$) of functions $h:[0,\infty)^{n+1} \to [0,\infty)$ defined by  
$$h(t_0,t_1,\ldots,t_{n})=t_0 ~\om~ t_1 ~\om \cdots \om ~t_n.$$

\subsection{Topology and Convergence}

Let $(X,d_\om,a)$ be an O-metric space. The $\om$-metric $d_\om$ allows the definition of:
\begin{enumerate}
\item a topology on $X$ called the $\om$-topology (or more broadly, the O-metric topology):
\begin{equation}
\mathcal{T}=\left\{A \subset X: ~ \forall x \in A ~\exists r>0,~  B(x,r) \subset A \right\},
 \label{topology}
\end{equation}
where the set $B(x,r):=\{y \in X: ~ |d_\om(x,y)-a|<r\}$ is called the open ball (although not necessarily an open set in the O-metric topology) centered on $x \in X$ and with radius $r>0$;

\item \textbf{O-convergent} sequences as sequences $\{x_n\}$ of points in $X$ such that 
\begin{equation}\label{O-conv} 
\exists x \in X \mbox{ such that } \displaystyle \lim_{n \to \infty} d_\om(x_n,x)=a \quad \mbox{ in $\mathbb{R}$},
\end{equation}
in which case $x$ is called the O-limit (or simply limit) of the sequence $(x_n)$ and we  write  $x_n \xrightarrow{\text{O}} x$;

\item Cauchy sequences as sequences  $\{x_n\}$ of points in $X$ such that:
\begin{equation}
\displaystyle \lim_{n,m \to \infty} d_\om(x_n,x_m)=a \quad \mbox{ in $\mathbb{R}$}.
\end{equation}

\item an \textbf{O-complete} O-metric space $(X,d_\om,a)$ as a space in which the Cauchy sequences are the  O-convergent sequences in $X$.
\end{enumerate}
\noindent
Whenever no further conditions are placed on the binary operation $\om$, these concepts do not follow laws on metric spaces: a sequence may have infinitely many O-limits, an O-convergent sequence may not even be a Cauchy sequence. However, every O-covergent sequence converges to each of its O-limits in the O-metric topology, hence, in the remaining part of the article, when no confusion arises, ``convergence" will mean ``O-convergence", ``limit" will mean ``O-limit" , and ``$x_n \to x$" will mean ``$x_n \xrightarrow{\text{O}} x$". In fact, the following holds:

\begin{proposition}[see \cite{ourpaper},\cite{ourpaper2}]
Let $(X,d_\om,a)$ be an O-metric space.
\begin{itemize}
\item[(i)] If $d_\om$ is $a$-upward, $a \, \om \, a =a$, and $\om$ is continuous  at $(a,a)$, then every O-convergent sequence is a Cauchy sequence.

\item[(ii)]
Convergent sequences in an O-metric space have unique limits if the following conditions simultaneously hold:
\begin{itemize}
\item[$(U_1)$]
$\om$ is continuous at points $(u,v)$ such that $u=a$ or $v=a$;
\item[$(U_2)$]
$\om$ is nondecreasing in both variables and either $u \, \om \, a =a \Leftrightarrow u=a$ for all $u \in I_a$, or $a \, \om \, u = a \Leftrightarrow u=a$ for all $u \in I_a$. 
\end{itemize}

\item[(iii)]
If $d_\om$ is $a$-upward, then every open ball is an open set if the following conditions simultaneously hold:
\begin{itemize}
\item[$(C_1)$] There exists $\gamma:[a,\infty) \times [a,\infty) \to \mathbb{R}$ such that $\gamma(r,u)>a$ and $\om(u,\gamma(r,u))\leq r$ for $u,r \geq a$ such that $u \in [a,r)$.
\item[$(C_2)$] $\om$ is increasing in both variables.
\end{itemize}
In such case, $X$ is also Hausdorff and O-convergence is equivalent to convergence in the O-topology.

\end{itemize}
\end{proposition}
\noindent
In the remaining part of the paper, we assume the following except when stated otherwise:
\begin{assumption}\label{mainass}
$(X,d_\om,a)$ is taken to be an $\om$-metric space, where $\om$ satisfies conditions $(U_1)$ and $(U_2)$, and $d_\om$ is $a$-upward. Therefore, every O-convergent sequence is a Cauchy sequence and has a unique O-limit.
\end{assumption}

\section{Generalized series}

To show that a sequence $\{x_n\}$ of points in an O-metric space  $(X,d_\om,a)$ is Cauchy, it suffices, via polygon inequalities, to show that
\[
\lim_{n,m \to \infty} h(d_\om (x_p,x_{p+1}), \ldots , d_\om(x_{q-1},x_q))=a, \quad \text{ with } p=\min\{n,m\}, ~ q=\max\{n,m\},
\]
for some $h \in \Omega_n$. 
This naturally leads to the study of the convergence of generalized $\om$-series to a limit determined by functions in $\Omega_n$ for large $n$. 
Examples demonstrate how classical additive and multiplicative series are recovered as special cases of generalized $\om$-series.

\subsection{Infinite generalized series}

Given a binary operation $\om$ on $\mathbb{R}$, and a sequence $\{t_n\}_{n \geq 0}$ of real numbers, the sequence $\{\omega^n\}_{n \geq 0}$ defined by the recursion
\begin{equation}
\left\{
\begin{array}{lll}
\omega^0 &=& t_0, 
\\
\omega^1 &=& t_0 \, \om \, t_1,
\\
\omega^n &=& h_{n+1}(t_0,t_1,t_2,\ldots,t_n), ~\mbox{ for } n \geq 2 \mbox{ and for some } h_{n+1} \in \Omega_n,
\end{array}
\right.
\end{equation}
 is called 
sequence of partial compositions of $\{t_n\}_{n \geq 0}$ following the pattern of functions $\{h_n\}_{n \in \mathbb{N}}$ and denoted $\omega^n= \som{0}{n} t_i.$ 
\\
The expression $\som{0}{\infty} t_i$ is called an infinite generalized series (or $\om$-series) following the  pattern of functions $\{h_n\}_{n \in \mathbb{N}}$.  When $\{\omega^n\}_{n \in \mathbb{N}}$ converges, we say that $\{t_n\}_{n \geq 0}$ is composable (or $\om$-composable) following the pattern of functions $\{h_n\}_{n \in \mathbb{N}}$, and write $$\lim_{n \to \infty} \omega^n =\som{0}{\infty} t_i.$$
\noindent
In the case of an associative binary operation $\om$, $\om$-series  can be written uniquely in the sense that the set $\Omega_n$ is a singleton for each integer $n \geq 2$. The following example can  easily be made as illustration:

\begin{example}
\textup{
Let $\{\omega^n\}_{n \geq 0}$ be the sequence of partial compositions of a sequence $\{t_n\}_{n \geq 0} \subset \mathbb{R}$, 
\begin{itemize}
\item[(a)] If $\om$ is the addition in $\mathbb{R}$ (i.e. when $\om(u,v)=u+v$), then $\om$-series are series in the usual sense: for $n \geq 0$,
$\som{0}{n} t_i=\displaystyle \sum_{i=0}^n t_i$;
\item[(b)] If $\om$ is the multiplication in $\mathbb{R}$ (i.e. when $\om(u,v)=uv$), then $\som{0}{\infty} t_i$ is the infinite product $\displaystyle \prod_{i=0}^\infty t_i$. 
\item[(c)] If $\om$ is the maximum function in $\mathbb{R}$, i.e. when $\om(u,v)=\max\{u,v\}$, then for $n \geq 0$, $\omega^n=\som{0}{n} t_i =\displaystyle \max_{0 \leq i \leq n} t_i$. In particular,
\begin{equation*}
\left\{
\begin{array}{lll}
\som{0}{\infty} t_i =\displaystyle \lim_{n \to \infty} t_n \mbox{ if the sequence of terms $\{t_n\}$ is nondecreasing};
\\
\som{0}{\infty} t_i=t_0 \mbox{ if $\{t_n\}$ is a non-increasing sequence}.
\end{array}
\right.
\end{equation*}
\end{itemize}
}
\end{example}
\noindent

\subsection{Patterned Iterations}

When $\om$ is not associative, it becomes necessary to interpret generalized $\om$-series as patterned compositions, providing a bridge between abstract series and concrete iterations.
\begin{definition}\label{pat}
Let $\om$ be a binary operation on $\mathbb{R}$, and let $\{\alpha_n\}_{n \in \mathbb{N}}$ be a sequence of non-negative integers, with $\alpha_n \in [0, n-1]$ for all $n \in \mathbb{N}$. 

\noindent
A sequence $\{h_n\}_{n \in \mathbb{N}}$ of functions $h_n \in \Omega_{n-1}$ is said to follow the pattern of integers $\{\alpha_n\}_{n \in \mathbb{N}}$ if
\begin{equation}
h_n(t_1,t_2,\ldots,t_n) = h_{\alpha_n}(t_1,t_2,\ldots,t_{\alpha_n}) \, \om \, h_{n-\alpha_n}(t_{\alpha_n+1},\ldots,t_n) ~~\forall n \geq 2.
\end{equation}
\noindent
If an $\om$-series $\som{0}{\infty} t_i$ following the pattern of functions $\{h_n\}_{n \in \mathbb{N}}$ is such that  $\{h_n\}_{n \in \mathbb{N}}$ follows the pattern of integers $\{\alpha_n\}_{n \in \mathbb{N}}$,  $\som{0}{\infty} t_i$ is also said to follow the pattern of integers $\{\alpha_n\}_{n \in \mathbb{N}}$. In such case,
\begin{equation*}
\som{1}{n}t_i = \left(\som{1}{\alpha_n} t_i \right) \, \om \, \left(\som{\alpha_n+1}{n} t_i  \right) =  
 \left(\som{1}{\alpha_n} t_i \right) \, \om \, \left(\som{1}{n-\alpha_n} t_{\alpha_n+i}  \right)
\end{equation*}
\end{definition}
\noindent
We state some examples below of infinite generalized series following some pattern of functions:
\begin{example}\label{doyouknowex}
\textup{
Let $g:[0,\infty) \to [0,\infty)$ be an additive function, in the sense that $g(u+v)=g(u)+g(v)$ for all $u,v \geq 0$. Consider the binary operation $\om(u,v)=g(u+v)$ on $[0,\infty)$. Let $\{t_n\}_{n \in \mathbb{N}}$ be a sequence of non-negative real numbers. 
\\
\textbf{1.} If $\som{1}{\infty}t_i$ is the $\om$-series following the pattern of integers $\{1\}_{n \in \mathbb{N}}$, then
\begin{equation*}
\som{1}{n+1}t_i  =\displaystyle \sum_{i=1}^{n}g^{(i)}(t_i)+ g^{(n)}(t_{n+1}).
\end{equation*}
In the particular case of the multiplication $g(u)=Lu$ by $L$, where $L >0$, 
\begin{equation}\label{doyouknow}
\som{1}{n+1}t_i=\displaystyle \sum_{i=1}^{n}L^{i} t_i+ L^{n}t_{n+1}.
\end{equation}
In this case, $\som{1}{\infty}t_i$ converges under any of the following conditions:
\begin{itemize}
\item[(i)] There is some $n_0 \in \mathbb{N}$ for which $t_n=0$ for $n \geq n_0$ (i.e. $\{t_n\}$ has a null tail); in fact, 
$\som{1}{\infty}t_i=\displaystyle \sum_{i=1}^{n_0}L^{i} t_i+ L^{n_0}t_{n_0+1}$.
\item[(ii)] $t_n=L^n$ for each $n$, and $L<1$, so that $\som{1}{n+1}t_i=\displaystyle \sum_{i=1}^{n}L^{2i}+ L^{2n+1}=\dfrac{L^2(1-(L^2)^{n+1})}{1-L^2}+L^{2n+1}$ and  $\som{1}{\infty}t_i =\dfrac{L^2}{1-L^2}$.
\end{itemize}
\noindent
\textbf{2.} If $\som{1}{\infty}t_i$ is the $\om$-series following the pattern of integers $\left\{2^{\ceil*{\log_2n}-1}\right\}_{n \in \mathbb{N}}$ (or the pattern of integers  $\left\{\ceil*{\frac{n}{2}}\right\}_{n \in \mathbb{N}}$), where $\ceil*{.}$ denotes the ceiling function, then for any $r \geq 1$, 
$$\som{1}{2^r}t_i = g^{(r)}\left(\displaystyle \sum_{j=1}^{2^r}t_j\right) =\displaystyle \sum_{j=1}^{2^r}  g^{(r)}(t_j).$$
}
\end{example}

\begin{example}\label{lemar}
\textup{
Let $\psi$ be a closed binary operation on the interval $[0,\infty)$ of non-negative real numbers. Let $u,v$ be non-negative real numbers, and let $\{t_n\}_{n \geq 1}$ be the sequence of real numbers such that
\begin{equation*}
t_n=\left\{
\begin{array}{lll}
\psi(u,v), & n=1
\\
v, & n \geq 2.
\end{array}
\right.
\end{equation*}
If $\spsi{1}{\infty} t_i$ is the infinite generalized series (or more precisely, the $\opsi$-series) following the pattern of integers $\{n-1\}_{n \in \mathbb{N}}$, then for $n \geq 2$,
\begin{equation*}
\spsi{1}{n}t_i=\psi\left(\spsi{1}{n-1}t_i, t_n\right)=\psi(\cdot,t_n)\left(\spsi{1}{n-1}t_i\right) = \psi(\cdot,v)\left(\spsi{1}{n-1}t_i\right) = \ldots = \psi^{(n)}(\cdot,v)(u),
\end{equation*}
where $\psi^{(n)}(\cdot,v)$ denotes the $n$-th iterate of the function $\psi(\cdot,v):[0,\infty) \to [0,\infty)$ defined for all $w \geq a$ by $\psi(\cdot,v)(w)=\psi(w,v)$. Therefore, $\spsi{1}{\infty}t_i=\displaystyle \lim_{n \to \infty}\psi^{(n)}(\cdot,v)(u)$. 
\\
Now, suppose that $\psi(\cdot,v)$ has a unique fixed point equal to $\displaystyle \lim_{n \to \infty}\psi^{(n)}(\cdot,v)(x_0)$ for any $x_0 \in [0,\infty)$ (for example, when $\psi(\cdot,v)$ is a $k$-contraction for some $k \in [0,1)$, following the Banach contraction principle). Then, 
\[\spsi{1}{\infty}t_i=\displaystyle \lim_{n \to \infty}\psi^{(n)}(\cdot,v)(u)=\lim_{n \to \infty}\psi^{(n)}(\cdot,v)(v)=\lim_{n \to \infty}\spsi{1}{n+1}v=\spsi{1}{\infty} v.\]
}
\end{example}

\section{Control Pairs and $\om$-Calculus compatible with the $\om$-Metric structure}

We introduce control pairs $(\varphi,\gamma)$ independently of any fixed point problem. The function $\varphi$ governs inequalities of the form
\begin{equation}\label{little}
u \le v \, \om \, \varphi(k,u),
\end{equation} 
while the function $\gamma$ prevents growth under the operation $\om$. The conditions introduced below ensure summability and compatibility with the O-metric structure.
\\

\noindent
In this section, and without prejudice to the earlier Assumption \ref{mainass} pertaining to binary operations associated to an O-metric, we consider the following:

\begin{assumption}\label{mainass1}
$a$ is a non-negative real number, $\om$ is a monotone (nondecreasing in both variables) closed binary operation  on $[a,\infty)$, and continuous at $(a,a)$, with $a \, \om \,  a=a$.
\end{assumption}


\begin{definition}\label{contraction2}
\textup{
Consider a (control) pair of functions $(\varphi,\gamma)$ as follows:
\begin{enumerate}
\item For the function $\varphi:[0,\infty) \times [a,\infty) \to [a,\infty)$, distinctive properties include: 
\begin{itemize}
\item[$(\varphi_0)$] $\varphi$ is continuous in the first variable at $0$.
\item[$(\varphi_1)$] $\varphi$ is continuous in the second variable at $a$, and increasing on both variables. 
\item[$(\varphi_2)$] $\varphi(0,u)=\varphi(r,a)=a$ for all $u \ge a$ and $r \ge 0$. 
\item[$(\varphi_3)$] $\varphi(r_1,\varphi(r_2,u)) =\varphi(r_1r_2,u)$ for all $r_1,r_2 \in [0,\infty)$ and $u \ge a$.
\item[$(\varphi_4)$]  $\varphi$ is distributive with respect to $\om$, i.e., 
$$\varphi(r, \,u \, \om \, v)=\varphi(r,u) \, \om \, \varphi(r,v)$$ for all $u,v \ge a$ and $r \ge 0$. 
\item[$(\varphi_5)$] For some $k \geq 0$,  the function $\psi(\cdot,v):[a,\infty) \to [a,\infty)$ defined for all $w \geq a$ by $$\psi(\cdot,v)(w)=\psi(w,v):=v \, \om \, \varphi(k,w)$$ has a unique fixed point $z=\displaystyle \lim_{n \to \infty} \psi(\cdot,v)^{(n)}(w)$ for any $w \geq a$. 
\item[$(\varphi_6)$]  For some $k \geq 0$, $\som{1}{\infty} \varphi(k^i,\cdot)$ following the pattern of integers $\{1\}_{n \in \mathbb{N}}$, is a well-defined function, continuous at $a$. 
\end{itemize}
\item For the function function  $\gamma:[a,\infty) \times [a,\infty) \to [a,\infty)$, properties include:
\begin{itemize}
\item[$(\gamma_1)$] $\gamma$ is continuous on both variables; 
\item[$(\gamma_2)$] $\max\{\gamma(u,v), \gamma(u \, \om \, v, a)\} \le \max\{u,v\}$ for all $u,v \ge a$. 
\end{itemize}
\end{enumerate}
}
\end{definition}
\noindent
It should be noted that conditions $(\varphi_0) - (\varphi_4)$ are relative to the continuity and initial conditions of the control function $\varphi$ at the boundaries $x=0$ and $y=a$, and its compatibility with the classical product via mixed associativity and distributivity. Conditions $(\varphi_5)$ and $(\varphi_6)$ ensure control of repeated $\om$-compositions of terms in $\varphi$ and $k$.

\begin{remark}\label{workeasy}
\textup{
One can easily check that if $\varphi$ satisfies conditions $(\varphi_1) - (\varphi_3)$, since $\om:[a,\infty) \times [a,\infty) \to [a,\infty)$ is nondecreasing in both variables, then 
\begin{equation}\label{acee}
\varphi(r,t) <\varphi(1,t) = t \quad \forall r \in [0,1) ~ \forall t \in [a,\infty).
\end{equation}
Therefore, 
\begin{equation}\label{extree}
\text{if $t \leq \varphi(r,t)$  with  $r \in [0,1)$ and $t \in [a,\infty)$, then $t=a$.}
\end{equation}
}
\end{remark}

\noindent
We present some examples to demonstrate the naturality and breadth of the framework.

\begin{example}\label{global} \textup{
Let $\om$ be a scalar augmented addition, i.e. $$\om(u,v):=s(u+v)$$ for all $u,v \geq 0$, for some $s \geq 1$, as in the case of b-metrics. The product function 
\begin{equation*}
\begin{array}{lclll}
 \varphi: & [0,\infty) \times [0,\infty) &\to& [0,\infty)
\\
&(r,u) & \mapsto & ru
\end{array}
\end{equation*}
satisfies conditions $(\varphi_0)-(\varphi_5)$  with $a =0$ and $C=s$. It also satisfies condition $(\varphi_6)$ for $k \in [0,1/s)$ since for any $v \geq 0$, the function $\psi(\cdot,v):[0,\infty) \to [0,\infty)$ defined for all $w \geq 0$ by $\psi(\cdot,v)(w)=\om(v,\varphi(k,w))=s(v+kw)$ is a $k$-contraction, with unique fixed point $\frac{sv}{1-ks} \in [0,\infty)$. We have that for $\epsilon \geq 0$ and $n \in \mathbb{N}$, following the pattern of integers $\{1\}_{n \in \mathbb{N}}$,
$$\som{1}{n+1} \varphi(k^i,\epsilon)=\som{1}{n+1} k^i\epsilon=\displaystyle \sum_{i=1}^n s^i k^i \epsilon +s^n k^{n+1} \epsilon =\left[ \sum_{i=1}^n (sk)^i +(sk)^nk \right]\epsilon,$$ 
hence $(\varphi_7)$ is satisfied for $k \in [0,\frac{1}{s})$ since $\som{1}{\infty} \varphi(k^i,\cdot)$ which is the function defined by $\som{1}{\infty} \varphi(k^i,\cdot)(\epsilon)=\dfrac{sk}{1-sk}\epsilon$, is continuous at $0$. 
\\
On the other hand, the function $$\gamma(u,v):=\frac{u+v}{2s} \text{ for all $u,v \geq 0$},$$ satisfies conditions $(\gamma_1)$ and $(\gamma_2)$, with $a=0$.
}
\end{example}
\begin{example}
\textup{
One can also check that, in the case of b-multiplicative metric spaces, where $\om$ is the $s$-powered product $$\om(u,v):=(uv)^s$$ with $s\geq1$ a constant, one can take $\varphi$ and $\gamma$ such that $$\varphi(r,u):=u^r$$ and $$\gamma(u,v):=(uv)^{1/2s}$$ for all $u,v \geq 1$ and $r \geq 0$. 
\\
Conditions $(\varphi_0)-(\varphi_6)$  and $(\gamma_1)-(\gamma_2)$ are all satisfied, for $k \in [0,\frac{1}{s})$.
}
\end{example}

\subsection{Separation Principle for $\om$-Inequalities}

A central difficulty in O-metric spaces arises from the fact that the binary operation $\om$ need not be associative and generally admits no algebraic inverse. Consequently, classical techniques based on subtraction or division cannot be applied directly to  inequalities of the type (\ref{little}).
\\

\noindent
The following theorem provides a mechanism for resolving such inequalities. In particular, it allows one to separate the variables appearing in expressions of the form (\ref{little}), thereby reducing them to controlled estimates governed by the associated $\om$-series. In this sense, the result plays a role analogous to some combined subtraction-division in classical additive settings.

\begin{theorem}\label{terrible}[Separation Principle]
Let $\varphi$ be the a control function satisfying the conditions  $(\varphi_i)$ for $1 \le i \le 5$ introduced above, for some $k \in (0,1)$. Suppose $u,v \ge a$ satisfy the inequality
\begin{equation}\label{hilal}
u \leq v \, \om \, \varphi(k,u).
\end{equation}
Then, the generalized series $\som{0}{\infty} \varphi(k^i, v)$ is composable following the pattern of integers $\{1\}_{n \in \mathbb{N}}$, and
\begin{equation}\label{hilal1}
u \leq  \som{0}{\infty} \varphi(k^i, v).
\end{equation}
\end{theorem}

\begin{proof}
(i) Define $\psi(u,v):=v \, \om \, \varphi(k,u)$, for all $u,v \in [a,\infty)$, for some $k \in (0,1)$.
Suppose $u,v$ are such that $u \leq v \, \om \, \varphi(k,u)$. By successive applications of the inequality $u \leq  v \,\om \, \varphi(k,u)=\psi(u,v)$, one obtains  $u \leq \spsi{1}{n} t_i$, where $t_1=\psi(u,v)$ and $t_i=v$ for $i \geq 2$, and where $\spsi{1}{n} t_i$ follows the pattern of integers $\{n-1\}_{n \in \mathbb{N}}$. Since $(\varphi_6)$ holds, we repeat the argument in Example \ref{lemar} to obtain as $n \to \infty$, the inequality 
\begin{equation}\label{refer}
u \leq \spsi{1}{\infty} t_i=\spsi{1}{\infty} v.
\end{equation}
(ii) We can prove by induction that $\varphi\left(k, \som{1}{n} \varphi(k^{i-1},v)\right)=\som{1}{n}\varphi(k^i,v)$ for all $v \in [a,\infty)$ and $n \in \mathbb{N}$, where the $\om$-series in the equality follow the pattern of integers $\{1\}_{n \in \mathbb{N}}$. Indeed, $\varphi\left(k, \som{1}{1} \varphi(k^{i-1},v)\right)=\varphi(k,\varphi(1,v))=\varphi(k,v)=\som{1}{1}\varphi(k^i,v)$ for all $v \in [a,\infty)$, and if one supposes that  $\varphi\left(k, \som{1}{m} \varphi(k^{i-1},v)\right)=\som{1}{m}\varphi(k^i,v)$ for all $v \in [a,\infty)$, for some $m \in \mathbb{N}$, then, for all $v \in [a,\infty)$,
\begin{equation*}
\begin{array}{lcl}
\varphi\left(k, \som{1}{m+1} \varphi(k^{i-1},v)\right) &=& \varphi\left(k, \, v \, \om  \left(\som{2}{m+1} \varphi(k^{i-1},v)\right)\right)
\\
&=& \varphi\left(k, \, v \, \om  \left(\som{1}{m} \varphi(k^{i},v)\right)\right)
\\
&=& \varphi\left(k, \, v \, \om  \left(\som{1}{m} \varphi(k^{i-1},\varphi(k,v))\right)\right)
\\
&=&  \varphi(k,v) \, \om \, \varphi\left(k, \som{1}{m} \varphi(k^{i-1},\varphi(k,v)) \right)  
\\
&=&  \varphi(k,v) \, \om \, \left( \som{1}{m}\varphi(k^i,\varphi(k,v))  \right)
\\
&=&  \varphi(k,v) \, \om \, \left( \som{1}{m}\varphi(k^{i+1},v)  \right)
\\
&=& \varphi(k,v) \, \om \left( \som{2}{m+1}\varphi(k^{i},v)  \right)
\\
&=&\som{1}{m+1}\varphi(k^{i},v).
\end{array}
\end{equation*}
(iii) Now, we prove by induction that $\spsi{1}{n} v= \som{1}{n} \varphi(k^{i-1},v)$ for all $v \in [a,\infty)$ and for all $n \in \mathbb{N}$, where the $\psi$-series at the left side of the inequality follows the pattern of integers $\{n-1\}_{n \in \mathbb{N}}$ while the $\om$-series at the right side follows the pattern of integers $\{1\}_{n \in \mathbb{N}}$. The equality can be easily verified to be true for $n=1$, and if one supposes that  $\spsi{1}{m} v= \som{1}{m} \varphi(k^{i-1},v)$ for all $v \in [a,\infty)$, for some $m \in \mathbb{N}$, then, for all $v \in [a,\infty)$,
\begin{equation*}
\begin{array}{llcccclll}
\spsi{1}{m+1} v &=& \psi\left( \spsi{1}{m} v, v   \right)   &=&  v \, \om\, \varphi\left(k, \spsi{1}{m} v \right) 
&=& v \, \om \, \varphi\left(k, \som{1}{m} \varphi(k^{i-1},v) \right) 
\\
&=& v \,  \om \, \left(\som{1}{m}\varphi(k^i,v) \right) &=& \varphi(1,v) \, \om \, \left(\som{1}{m}\varphi(k^i,v) \right)

&=& \som{1}{m+1} \varphi(k^{i-1},v).
\end{array}
\end{equation*}
Therefore, for $u,v \in [a,\infty)$ such that $u \leq  v \, \om \, \varphi(k,u)=\psi(u,v)$, inequality (\ref{refer}) becomes:
\begin{equation*}
u \leq \spsi{1}{\infty} v =\som{1}{\infty} \varphi(k^{i-1},v)=\som{0}{\infty} \varphi(k^{i},v).
\end{equation*}
\end{proof}
\noindent
The theorem provides a general mechanism for resolving $\om$-inequalities generated by non-associative compositions and will serve as the analytic foundation for the fixed point results established in the sequel.
\\

\begin{example}
As seen in Example \ref{global}, given a fixed real number $s \geq 1$, the operations $u \om v :=s(u+v)$ and $\varphi(k,u):=ku$ satisfy the conditions of Theorem \ref{terrible} for $k \in \left(0,\frac{1}{s}\right)$, and
$$\som{1}{\infty} \varphi(k^i,\epsilon)=\dfrac{s\epsilon}{1-sk}.$$
The inequality (\ref{hilal}) becomes
\begin{equation}\label{awel}
u \leq s(v + ku),
\end{equation}
for $u,v \geq 0$, which, when solved, yields
$$u \le \frac{sv}{1-sk}.$$
The inequality above is the same inequality as in  (\ref{hilal1}).
\end{example}

\noindent
The following consequence describes the asymptotic behaviour of quantities governed by $\om$-compositions in the presence of vanishing perturbations. In particular, it shows that when the perturbation sequence $(q_n)$ approaches the minimal level 
$a$, any compatible limiting behaviour of $(p_n)$ must collapse to $a$.

\begin{corollary}\label{Ibe}
Suppose $\om$ is a closed binary operation on $[a,\infty)$. Let $(p_n)$ and $(q_n)$ be sequences of real numbers not less than $a$ such that
\begin{equation}
\label{molomolo11}
\left\{
\begin{array}{lll}
p_n \to p, 
\\
q_n \to a,
\\
p \, \le q_n \, \om \, p_n,
\\
p_{n+1} \le \varphi\left(k,\max\{p,p_n\}\right)
\end{array}
\right.
\end{equation}
where $\varphi$ is a function satisfying conditions $(\varphi_i)$ for all $i$.
Then, $p=a$.
\end{corollary}

\begin{proof}
We prove by cases.

\noindent
Case 1: Suppose there is $m \geq n_0$ such that $p_m<p$. 
\\
In this case, $p_{m+1} \leq \varphi(k,p)$, hence, for any $n > m$, 
\begin{equation}\label{intermed1}
p_n \leq \varphi(k,p).
\end{equation}
In fact, for $n>m$,
$$p_n \leq \varphi(k, \, q_n\, \om \, p_n) = \varphi(k,q_n)) \, \om \, \varphi(k,p_n),
$$
hence from Theorem \ref{terrible}, 
$$
p_n \leq \som{0}{\infty}\varphi(k^i,\varphi(k,q_n)) = \som{0}{\infty}\varphi(k^{i+1},q_n)
= \som{1}{\infty}\varphi(k^{i},q_n),$$
where $\som{1}{\infty}\varphi(k^{i},q_n)$ follows the pattern $\{1\}_{n \in \mathbb{N}}$ of integers.
As $n \to \infty$, we have by $(\varphi_5)$ that $\som{1}{\infty}\varphi(k^{i},q_n) \to \som{1}{\infty}\varphi(k^{i},a)= a$; thus  $p_n \to a$.
\\
Case 2: Assume now that $p \leq p_n$ for all $n \geq n_0$, then from (\ref{molomolo11}), 
\begin{equation}\label{encore1}
p_{n+1} \leq \varphi(k, p_n),
\end{equation}
which, combined with $(\varphi_3)$ gives
\begin{equation}
p_n \le \varphi(k^{n-n_0},p_{n_0}).
\end{equation}
As $n \to \infty$, $\varphi(k^{n-n_0}, p_{n_0}) \to a$ hence $p_n \to a$.
\end{proof}

\subsection{Interval of Cauchyness of contractive sequences}

By virtue of Remark \ref{workeasy}, it turns out that the control function $\varphi$ is a contracting function. Therefore, the following makes sense in the context of O-metric spaces satisfying Assumption \ref{mainass}: 

\begin{definition}
Let $(X,d_\om,a)$ be an $\om$-metric space. 
Let $\varphi:[0,\infty) \times [a,\infty) \to [a,\infty)$ be a function satisfying properties $(\varphi_1)$ -- $(\varphi_3)$.
 Given a real number $k \in [0,1)$, a sequence  $\{x_n\}$ of points of $X$ such that 
\[
d_{\om}(x_{n}, x_{n+1}) \leq \varphi(k, d_\om(x_{n-1}, x_{n})) \text{ for all }n \in \mathbb{N}
\]
is called a $(k,\varphi)$-contracting sequence.
\end{definition}

\noindent
The following lemma provides a bridge between the convergence properties of the generalized $\om$-series and the behaviour of iterative sequences in O-metric spaces: it links the range $\mathcal{C}_\varphi$ of convergence to base level $a$ of an O-series $\som{n \wedge m}{n\vee m} \varphi(r^i,\epsilon)$ (following some pattern of compositions) generated by a control function $\varphi$  to the threshold that guarantees the Cauchy property of contracting sequences.

\begin{lemma}\label{FrmOP}[see \cite{ourpaper2}]
Let $\om$ be a binary operation satisfying Assumption \ref{mainass1}.  
Let $\varphi:[0,\infty) \times [a,\infty) \to [a,\infty)$ be a function satisfying $(\varphi_1) - (\varphi_3)$. Define the set $\mathcal{C}_\varphi$  by
\begin{equation}\label{conver}
\mathcal{C}_\varphi=\left\{
r \geq 0 : ~~ \forall \epsilon \geq a ~~   \lim_{n,m \to \infty} \som{n \wedge m}{n\vee m} \varphi(r^i,\epsilon) = a 
\right\},
\end{equation}
where $n \wedge m:=\min\{n,m\}$ and $n \vee m:=\max\{n,m\}$.
\\
The following hold:
\begin{enumerate}
\item $\mathcal{C}_\varphi$ is an interval, and $0 \in \mathcal{C}_\varphi \subset [0,1)$. 
\item\label{go} Any $k$-contracting sequence $\{x_n\}$ in an O-metric space $(X,d_\om,a)$ satisfying Assumption \ref{mainass},
 for which $k<\kappa$, where $\kappa:=\sup \mathcal{C}_\varphi$, is a  Cauchy sequence.
\end{enumerate}
\end{lemma}
\begin{proof}
If $r=0$, then for $n,m \in \mathbb{N}$, $\varphi(r^i,\epsilon)=\varphi(0,\epsilon)=a$ for all $\epsilon \geq a$ and $i \in \mathbb{N}$ such that $n \wedge m \leq i \leq n \vee m$, hence $\som{n \wedge m}{n \vee m} \varphi(r^i,\epsilon)=a$ no matter how the $\varphi(r^i,\epsilon)$ are composed. Therefore $0 \in C_\varphi$.
\\
Suppose $r \in C_\varphi$ and $s \in [0,r]$. Since $\varphi$ is nondecreasing in the first variable, $\varphi(s^i,\epsilon) \leq \varphi(r^i,\epsilon)$ for $i \in \mathbb{N}$ such that $n \wedge m \leq i \leq n \vee m$, given $n,m \in \mathbb{N}$. As $\om$ is nondecreasing in both variables, a map $(t_1,t_2,\ldots,t_j) \mapsto \som{1}{j}t_i$ is nondecreasing in all variables, so $\som{n\wedge m}{n \vee m}\varphi(s^i,\epsilon) \leq \som{n \wedge m}{n \vee m} \varphi(r^i,\epsilon)$ for any given pattern of compositions.  Thus $s \in C_\varphi$ and $C_\varphi$ is an interval.
\\
If $r=1$, then for $n \in \mathbb{N}$ and $\epsilon>a$, $\som{n}{n} \varphi(r^i,\epsilon)=\varphi(1, \epsilon) >a$. Thus $1 \notin C_\varphi$. 
\\
\\
Now, let $\{x_n\}$ be a sequence such that 
$d_{\om}(x_{n}, x_{n+1}) \leq \varphi(k, d_\om(x_{n-1}, x_{n}))$ for all $n \in \mathbb{N}$, where $k \in [0, \kappa)$ and $\kappa:=\sup \mathcal{C}_\varphi$.
From the triangle $\om$-inequality,  given $n,m \in \mathbb{N}$ such that $m >n$, $$d_\om(x_n,x_{m}) \leq  \som{n}{m -1}d_\om(x_i,x_{i+1})$$ no matter the way of composition. Since maps $(t_1,t_2,\ldots,t_j) \mapsto \som{1}{j}t_i$ are nondecreasing in all variables, and since $d_\om(x_n,x_{n+1}) \leq \varphi(k^n,d_\om(x_0,x_1))$ holds for all $n \in \mathbb{N}$, we have that $d_\om(x_n,x_{m}) \leq  \som{n}{m-1} \varphi(k^i,d_\om(x_0,x_1))$ for all $n, m \in \mathbb{N}$, with $m>n$. As $k \in [0,\kappa)$, $k \in \mathcal{C}_\varphi$ hence  $\displaystyle \lim_{n,m \to \infty} d_\om(x_n,x_{m})= \lim_{n,m \to \infty} \som{n \wedge m}{n \vee m} \varphi(r^i,d_\om(x_0,x_1)) = a$. Therefore  $\{x_n\}$ is a Cauchy sequence. 
\end{proof}
\noindent
From the lemma, we may view the set $\mathcal{C}_\varphi$ as the interval of ``Cauchyness" of $(k,\varphi)$-contracting sequences. 
\\
\\
It should also be emphasized that $r \in C_\varphi$ whenever  $\displaystyle  \lim_{n,m \to \infty} \som{n \wedge m}{n \vee m} \varphi(r^i,\epsilon) = a$ for all $\epsilon \geq a$, for at least one way (not necessarily every way) of composition of the  $\varphi(r^i,\epsilon)$.
For specific patterns, we define the set 
\begin{equation}\label{omg}
\mathcal{C}_\varphi^{\{\alpha_n\}}=\left\{ r \geq 0 : ~~
\begin{array}{lll}
\forall \epsilon \geq a ~~   \displaystyle \lim_{n,m \to \infty} \som{n \wedge m}{n\vee m} \varphi(r^i,\epsilon) = a, \mbox{ where }
\\
\mbox{ $\som{n \wedge m}{n \vee m} \varphi(r^i,\epsilon)$ follows the pattern of integers $\{\alpha_n\}$} 
\end{array}
\right\}.
\end{equation}
It is obvious that for any pattern $\{\alpha_n\}$ of integers, $\mathcal{C}_\varphi^{\{\alpha_n\}}$ is an interval, and 
$$0 \in \mathcal{C}_\varphi^{\{\alpha_n\}} \subset \mathcal{C}_\varphi, \quad  \mathcal{C}_\varphi =\bigcup \left\{ \mathcal{C}_\varphi^{\{\alpha_n\}}: \, \, \{\alpha_n\} \text{ is a pattern of integers.} \right\}.$$
\begin{example}\label{hmm}
\textup{
As seen in equality (\ref{doyouknow}) in Example \ref{doyouknowex}, if we take $\om(u,v)=L(u+v)$ where $L>0$ (which means $a=0$), then following the pattern of integers $\{1\}_{n \in \mathbb{N}}$, $\som{1}{n+1}t_i=\displaystyle \sum_{i=1}^{n}L^{i} t_i+ L^{n}t_{n+1}$, where $\{t_n\}$ is a sequence of nonnegative numbers. Therefore, for $n,j \geq 1$, $\som{n}{n+j}t_i=\displaystyle \sum_{i=1}^{j}L^{i} t_{i+n-1}+ L^{j}t_{n+j}$. Now, let $\varphi(r,u)=ru$, for $r,u \geq 0$. We have that:
\begin{equation*}
\begin{array}{lcl}
\som{n}{n+j} \varphi(r^i,\epsilon) &=& \som{n}{n+j} r^i \epsilon = \displaystyle \sum_{i=1}^{j} L^{i}r^{i+n-1}\epsilon +L^{j}r^{n+j}\epsilon
\\
&=& r^n \epsilon \left[\displaystyle \sum_{i=1}^{j} L^{i}r^{i-1} +L^{j}r^{j}  \right]
\\
&=& r^n \epsilon \left[\displaystyle L\sum_{i=1}^{j} (Lr)^{i-1} +(Lr)^{j}  \right] \to 0
\end{array}
\end{equation*}
as $n,j \to \infty$, provided $r<\min\{1,L^{-1}\}$. 
\\
Obviously, $\som{n}{n} \varphi(r^i,\epsilon)=r^n\epsilon \to 0$ as $n \to \infty$, provided $r<1$.   Hence
$\mathcal{C}_\varphi^{\{1\}} \supset \left[0,\min\left\{\frac{1}{L},1\right\}\right)$.
\\
Now, let $r=\frac{1}{L}$.  For $n,j \geq 1$, $\som{n}{n+j} \varphi(r^i,\epsilon) = r^n \epsilon \left[Lj +1 \right]$. 
In particular, 
$\som{n}{n+\ceil{L+1}^n} \varphi(r^i,\epsilon) = \left(\frac{\ceil{L+1}}{L}\right)^n\epsilon L+\frac{\epsilon}{L^n}$, which diverges to $\infty$. Thus $\frac{1}{L} \notin \mathcal{C}_\varphi^{\{1\}}$.
\\
If $r=1$ and $\epsilon>0$, then $\som{n}{n} \varphi(r^i,\epsilon) =\epsilon \nrightarrow 0$, hence $1 \notin \mathcal{C}_\varphi^{\{1\}}$.
\\
Therefore $\min\left\{\frac{1}{L},1\right\} \notin \mathcal{C}_\varphi^{\{1\}}$, and so $\mathcal{C}_\varphi^{\{1\}} = \left[0,\min\left\{\frac{1}{L},1\right\}\right)$.
\\
As established in \cite{ourpaper2}, $C_\varphi=[0,1)$ hence $\mathcal{C}_\varphi^{\{1\}} \subsetneq C_\varphi$ if $L>1$.
}
\end{example}
\noindent

\noindent
In the next section, we apply the tools developed in the current section to recursive inequalities generated by contractive mappings. By combining the separation principle with the convergence properties of generalized $\om$-series, one obtains criteria ensuring that the orbits of such mappings form Cauchy sequences, as well as conditions under which their limits (which would exist when the O-metric space is O-complete) are fixed points. 

\section{Application to Fixed Points}

We begin by stating as a direct consequence of Corollary \ref{Ibe} above, a lemma which specifies recursive relations on self-mappings on an O-metric space that ensure that the limit of a convergent Picard sequence is a fixed point.

\begin{lemma}\label{Godwillmake}
Let $T: X \to X$ be a mapping defined on an O-metric space $(X,d_\om,a)$, and let $x \in X$ be a given point. Consider a map $\varphi:[0,\infty) \times [a,\infty) \to [a,\infty)$ and a real constant $k \in (0,1)$ satisfying conditions $(\varphi_i)$ for all $i$, and suppose that the sequence $\{T^nx\}$ converges to some $u \in X$.
\begin{enumerate}
\item If there is some $n_0 \in \mathbb{N}$ such that 
\begin{equation}\label{molomolo1}
d(T^{n+1}x,Tu) \leq \varphi(k,\max\{ d_\om(u,Tu),d_\om(T^nx,Tu)\})
\end{equation}
for all $n \geq n_0$, then $u$ is a fixed point of $T$. 

\item If, for all $n \in \mathbb{N}$, we have that
\begin{equation}\label{terre}
d(T^{n+1}x,Tu) \le \varphi(k, \max \mathcal{R}_{T^nx,u}),
\end{equation}
where
$$\mathcal{R}_{T^nx,u} =\{d(T^nx,u), d(T^nx,T^{n+1}x),d(u,Tu),d(u,T^{n+1}x),d(T^nx,Tu)\},$$
then $u$ is a fixed point of $T$.

\item If $T$ is sequentially continuous, then $u$ is a fixed point of $T$.
\end{enumerate}
\end{lemma}

\begin{proof}
For the first statement, Corollary \ref{Ibe} applies easily, by taking $p_n=d_\om(T^n x, Tu)$ and $q_n=d(u,T^nx)$ for all $n$, and $p=d_\om(u,Tu)$. Since $d_\om(T^n x, Tu) \to a$, it follows that $u=Tu$, so $u$ is a fixed point of $T$.
\\
For the second statement, suppose $u \neq Tu$. In such case, $d(u,Tu)>a$ hence there is $n_0 \in \mathbb{N}$ such that
$$\max\{d_\om(x_n,x_{n+1}), d_\om(x_n,u), d_\om(x_{n+1},u)\} <d_\om(u,Tu)$$
for all $n \geq n_0$. Therefore, from inequality (\ref{terre}),
$$d(T^{n+1}x,Tu) \le \varphi(k, \max\{d(u,Tu),d(T^nx,Tu)\}$$
holds for all $n \ge n_0$. From the first statement of the lemma, $u$ is a fixed point of $T$, which contradicts the assumption
that $u \neq Tu$. Therefore, $u=Tu$ so, $u$ is indeed a fixed point of $T$.
\\
If $T$ is sequentially continuous, then $T^{n+1}x=TT^nx \to Tu$ as $n \to \infty$, hence $Tu=u$.
\end{proof}

\noindent
We can now state our first theorem which is somewhat an extension, in the context of O-metrics, of  Ciric's (in \cite{Ciric}).

\begin{theorem}\label{ZZthm}
Let $(X,d_\om,a)$ be an O-complete O-metric space satisfying Assumption \ref{mainass}. Let $T:X \to X$ be a map such that for all $x,y \in X$,
\begin{equation}\label{help2}
d_\om(Tx,Ty) \leq \varphi(k, \max\{d_\om(x,y), d_\om(x,Tx), d_\om(y,Ty), \gamma(d_\om(x,Ty), d_\om(y,Tx))\}),
\end{equation}
where $\varphi$ and $\gamma$ satisfy conditions $(\varphi_1) - (\varphi_3)$, and $(\gamma_1) - (\gamma_2)$ respectively, and $k<\kappa$, with $\kappa:=\sup C_\varphi$.
Then $T$ has a unique fixed point if any of the following conditions below holds:
\begin{itemize}
\item[(i)] $d_\om$ is sequentially continuous in the first variable;
\item[(ii)] $T$ is sequentially continuous;
\item[(iii)] $\varphi$ also satisfies condition $(\varphi_0)$ and $(\varphi_4)-(\varphi_6)$ for $k$.
\end{itemize}
\end{theorem}

\begin{proof} Let $x_0 \in X$ and $\{x_n\}_{n \in \mathbb{N}}$ be a sequence of points in $X$ such that $x_{n+1}=Tx_n$ for $n \geq 0$. For each
 $n \in \mathbb{N}$, $d_\om(x_n,x_{n+1}) = d_\om(Tx_{n-1},Tx_n) \leq \varphi(k,\max \mathcal{A}_{x_n})$, where 
\begin{equation*}
\begin{array}{lcl}
\mathcal{A}_{x_n} &:= & \left\{d_\om(x_{n-1},x_n), d_\om(x_{n-1},x_n), d_\om(x_n,x_{n+1}), \gamma(d_\om(x_{n-1},x_{n+1}), d_\om(x_n,x_n))\right\}
\\
&=&\left\{d_\om(x_{n-1},x_n), d_\om(x_n,x_{n+1}), \gamma(d_\om(x_{n-1},x_{n+1}), a)\right\}.
\end{array}
\end{equation*}
Now, $\gamma(d_\om(x_{n-1},x_{n+1}), a) \leq \gamma(d_\om(x_{n-1},x_{n}) \, \om \, d_\om(x_{n},x_{n+1}), a)$.
\\
If $d_\om(x_n,x_{n+1}) = a$, then $x_n = x_{n+1}$ and $x_n$ is a fixed point of $T$. 
\\
Suppose now that $d_\om(x_n,x_{n+1}) > a$. 
If $d_\om(x_n,x_{n+1})  \geq d_\om(x_{n-1},x_n)$, then, by $(\gamma_2)$, 
$$\gamma(d_\om(x_{n-1},x_{n+1}), a) \leq \gamma(d_\om(x_{n},x_{n+1}) \, \om \, d_\om(x_{n},x_{n+1}), a) \leq d_\om(x_n,x_{n+1})$$ and
\begin{equation} 
\begin{array}{lcl}
d_\om(x_n,x_{n+1}) &\leq& \varphi(k,\max \mathcal{A}_{x_n}) \leq  \varphi(k, \max\{d_\om(x_{n-1},x_n), d_\om(x_n,x_{n+1}), d_\om(x_{n},x_{n+1})\})  
\\&=&\varphi(k, d_\om(x_n,x_{n+1}))<d_\om(x_n,x_{n+1}),
\end{array}
\end{equation}
which is a contradiction. Hence  $d_\om(x_n,x_{n+1})  \leq d_\om(x_{n-1},x_n)$,
$$\gamma(d_\om(x_{n-1},x_{n+1}), a) \leq \gamma(d_\om(x_{n-1},x_{n}) \, \om \, d_\om(x_{n-1},x_{n}), a) \leq d_\om(x_{n-1},x_{n}),$$
and
\begin{equation*}
\begin{array}{lcl}
d_\om(x_n,x_{n+1}) &\leq& \varphi(k,\max \mathcal{A}_{x_n})  \\ &\leq& \varphi(k, \max\{d_\om(x_{n-1},x_n), d_\om(x_n,x_{n+1}),d_\om(x_{n-1},x_{n})\})
\\
&=& \varphi(k,d_\om(x_{n-1},x_n)).
\end{array}
\end{equation*}
Therefore, for $n \in \mathbb{N}$,
\begin{equation*}
\begin{array}{lcl}
d_\om(x_n,x_{n+1}) & \leq & \varphi(k,d_\om(x_{n-1},x_n)).
\end{array}
\end{equation*}
From Lemma \ref {FrmOP}, $\{x_n\}$ is a Cauchy sequence thus it converges to some point $u \in X$.\\\\
For all $n \in \mathbb{N}$,  
\begin{equation}\label{grace}
\left|
\begin{array}{lll}
d_\om(x_{n+1},Tu) =d_\om(Tx_n,Tu) \leq \varphi(k,\max \mathcal{B}_{x_n,u}), \mbox{ where}
\\
\mathcal{B}_{x_n,u}:=\{d_\om(x_{n},u), d_\om(x_{n},x_{n+1}), d_\om(u, Tu), \gamma(d_\om(x_{n},Tu), d_\om(u,x_{n+1}))\}.
\end{array}
\right.
\end{equation}
\noindent
If $d_\om$ is sequentially continuous in the first variable, then as $n \to \infty$, $d_\om(u,Tu) \leq \varphi(k,d_\om(u,Tu))$ hence $d_\om(u,Tu)=a$ and $u=Tu$. 
\\
\\
If (ii) or (iii) hold, then by Lemma \ref{Godwillmake}, $u$ is a fixed point of $T$. Uniqueness follows from (\ref{help3}) and (\ref{extree}).
\end{proof}

\begin{corollary}
Let $(X,d,s)$ be a complete b-metric space, with $s \geq 1$, and $T:X \to X$ a mapping such that for some $k \in (0,1)$, and for all $x,y \in X$,
\begin{equation}
d(Tx,Ty) \leq k \max\left\{d(x,y), d(x,Tx), d(y,Ty), \frac{d(x,Ty)+d(y,Tx)}{2s} \right\}.
\end{equation}
Then $T$ has a unique fixed point. 
\end{corollary}
\begin{proof}
The map $T$ satisfies condition (\ref{help2}) with $\varphi:[0,\infty) \times [0,\infty) \to [0,\infty)$ defined by $\varphi(r,u)=ru$ for all $r,u \geq 0$ and $\gamma:[a,\infty) \times [a,\infty) \to [a,\infty)$ defined by $\gamma(u,v)=\frac{u+v}{2s}$ for all $u,v \geq 0$. The function $\varphi$ satisfies conditions $(\varphi_0) - (\varphi_6)$ with $k \in [0,\frac{1}{s})$. The function $\varphi$ satisfies the conditions $(\gamma_1) - (\gamma_2)$. From Proposition 2.1 in \cite{ourpaper2}, $\kappa=1$ for b-metric spaces hence from Theorem \ref{ZZthm}, $T$ has a unique fixed point. 
\end{proof}

\noindent
Next, we replace completeness with the notion of orbital completeness which in turn yields a more general contractive-like condition than (\ref{help2}). 

\begin{definition}
Given a map $T:X \to X$, the O-metric space $X$ is said to be $T$-orbitally complete if every Cauchy sequence which is contained in the orbit $\mathcal{O}(x,\infty):=\{x,Tx,T^2x,\ldots\}$ of $T$ at some point $x \in X$, converges in $X$.
\end{definition}

\noindent
The following theorem is obtained:
\begin{theorem}\label{AAthm}
Let $(X,d_\om,a)$ be an O-metric space satisfying Assumption \ref{mainass}. 
Consider a map $T:X \to X$ such that for all $x,y \in X$,
\begin{equation}\label{help3}
d_\om(Tx,Ty) \leq \varphi(k, \max\{d_\om(x,y), d_\om(x,Tx), d_\om(y,Ty), d_\om(x,Ty), d_\om(y,Tx)\}),
\end{equation}
where $\varphi$ satisfy conditions $(\varphi_0) - (\varphi_5)$ with $k<\kappa$, where $\kappa:=\sup C_\varphi$.
Suppose $X$ is $T$-orbitally complete. Then $T$ has a unique fixed point if any of the following conditions below holds:
\begin{itemize}
\item[(i)] $d_\om$ is sequentially continuous in the first variable;
\item[(ii)] $T$ is sequentially continuous;
\item[(iii)] $\varphi$ also satisfies condition $(\varphi_6)$ for $k$.
\end{itemize}

\end{theorem}
\begin{proof} Let $x \in X$, $n \in \mathbb{N}$ and $i, j \in \{1,2,\ldots,n\}$. From (\ref{help3}) we have that:
\begin{equation}\label{okayo}
\begin{array}{lcl}
d_\om(T^ix,T^jx) &\leq& \varphi
\left(k, 
\max\left\{
\begin{array}{lll}
d_\om(T^{i-1}x,T^{j-1}x), d_\om(T^{i-1}x,T^ix), d_\om(T^{j-1}x,T^jx), 
\\
d_\om(T^{i-1}x,T^jx), d_\om(T^{j-1}x,T^ix)
\end{array}
\right\}
\right)
\\
&\leq&  \varphi(k,\delta(\mathcal{O}(x,n))),
\end{array}
\end{equation}
where $\mathcal{O}(x,n):=\{x,Tx,T^2x,\ldots,T^nx\}$, with $\delta(A):=\sup\{d_\om(x,y): x,y \in A\}$ for non-empty set $A$. In particular, for $i=1$ and $j=n$, we have that for all $x \in X$ and $n \in \mathbb{N}$,
\begin{equation}\label{change}
d_\om(Tx,T^nx) \leq \varphi(k,\delta(\mathcal{O}(x,n))).
\end{equation}
Since $\mathcal{O}(x,n)$ is finite and since $d_\om(T^ix,T^jx) < \delta(\mathcal{O}(x,n))$ (from inequality (\ref{okayo}) and Remark \ref{workeasy}), it follows that:
\begin{equation}\label{harder}
\forall n \in \mathbb{N}~ ~ \exists m \in \{1,2,\ldots,n\}: ~ \delta(\mathcal{O}(x,n))=d(x,T^mx).
\end{equation}
For such $m$, we have that:
\begin{equation}\label{thankGod}
\begin{array}{lcl}
d_\om(x,T^mx)  &\leq&   d_\om(x,Tx) \, \om \,  d_\om(Tx,T^mx) 
\\
&\leq& d_\om(x,Tx) \, \om \, \varphi(k, \delta(\mathcal{O}(x,n)))
\\
& \leq & d_\om(x,Tx) \, \om \, \varphi(k, d_\om(x,T^mx)).
\end{array}
\end{equation}
From Lemma \ref{terrible}, we have that:
\begin{equation}\label{willy}
d_\om(x,T^mx) \leq \som{0}{\infty} \varphi(k^i, d_\om(x,Tx)).
\end{equation}
Therefore $\delta(\mathcal{O}(x,n))=d(x,T^mx) \leq \som{0}{\infty} \varphi(k^i, d_\om(x,Tx)).$ Now, the sequence $\{\delta(\mathcal{O}(x,n))\}_{n \in \mathbb{N}}$  is increasing, as the sets $\mathcal{O}(x,n)$ form an increasing sequence. Thus 
\begin{equation}\label{canit}
\delta(\mathcal{O}(x,\infty))=\sup_{n \in \mathbb{N}}    \delta(\mathcal{O}(x,n)) \leq \som{0}{\infty} \varphi(k^i, d_\om(x,Tx)),
\end{equation}
where  $\mathcal{O}(x,\infty):=\{x,Tx,T^2x,\ldots\}$ is the orbit of $T$ at $x$. 
\\
For $n, m \in \mathbb{N}$ with $n<m$, we have from (\ref{change}) that:
\begin{equation}\label{lone}
\begin{array}{lcl}
d_\om(T^nx,T^mx) &=& d_\om(TT^{n-1}x,T^{m-n+1}T^{n-1}x) 
\\
&\leq& \varphi(k,\delta(\mathcal{O}(T^{n-1}x,m-n+1 ))).
\end{array}
\end{equation}
From (\ref{harder}), there is $m_1 \in \{1,2,\ldots,m-n+1\}$ such that $\delta(\mathcal{O}(T^{n-1}x,m-n+1 ))=d_\om(T^{n-1}x,T^{m_1}T^{n-1}x)$ hence, from (\ref{lone}),
\begin{equation}\label{lonea}
d_\om(T^nx,T^mx)  \leq \varphi(k, d_\om(T^{n-1}x,T^{m_1}T^{n-1}x)).
\end{equation}
Another application of (\ref{change}) yields:
\begin{equation}
\begin{array}{lcl}
d_\om(T^{n-1}x,T^{m_1+n-1}x) &=& d_\om(TT^{n-2}x,T^{m_1+1}T^{n-2}x) 
\\
&\leq& \varphi(k,\delta(\mathcal{O}(T^{n-2}x,m_1+1 )))
\\
&\leq& \varphi(k, \delta(\mathcal{O}(T^{n-2}x,m-n+2))),
\end{array}
\end{equation}
which combined with inequality (\ref{lonea}) gives
\begin{equation}\label{process1}
\begin{array}{lcl}
d_\om(T^nx,T^mx)  &\leq& \varphi(k, d_\om(T^{n-1}x,T^{m_1}T^{n-1}x)).
\\
&\leq&  \varphi(k, \varphi(k, \delta(\mathcal{O}(T^{n-2}x,m-n+2))))
\\
&=& \varphi(k^2,\delta(\mathcal{O}(T^{n-2}x,m-n+2))).
\end{array}
\end{equation}
Repeating the process which allowed us to obtain (\ref{process1}) from (\ref{lone}), we have that:
\begin{equation*}
\begin{array}{lcl}
d_\om(T^nx,T^mx)  &\leq& \varphi(k,\delta(\mathcal{O}(T^{n-1}x,m-n+1 )))
\\
&\leq& \varphi(k^2,\delta(\mathcal{O}(T^{n-2}x,m-n+2)))
\\
& \vdots & 
\\ 
&\leq & \varphi(k^n,\delta(\mathcal{O}(x,m))).
\end{array}
\end{equation*}
Therefore,
\begin{equation*}
\begin{array}{lcl}
d_\om(T^nx,T^mx) &\leq & \varphi(k^n,\delta(\mathcal{O}(x,m)))
\\
& \leq & \varphi(k^n,\delta(\mathcal{O}(x,\infty)))
\\
&\leq & \varphi\left(k^n, \som{0}{\infty} \varphi(k^i, d_\om(x,Tx))\right),
\end{array}
\end{equation*}
from (\ref{canit}). As $n \to \infty$, given the continuity of $\varphi$ in the first variable at $0$, we obtain $d_\om(T^nx,T^mx) \to a$, hence $\{T^nx\}$ is a Cauchy sequence; $X$ being $T$-orbitally complete,  $\{T^nx\}$ has a limit, say $u$, in $X$. From (\ref{help3}), if $n \in \mathbb{N}$,
$$d_\om(T^{n+1}x,Tu) \leq \varphi(k, \max\{ d_\om(T^nx,u), d_\om(T^nx,T^{n+1}x), d_\om(u,Tu),d_\om(T^nx,Tu), d_\om(u,T^{n+1}x) \}).$$
If $d_\om$ is sequentially continuous in the first variable, then as $n \to \infty$, $d_\om(u,Tu) \leq \varphi(k,d_\om(u,Tu))$ hence $d_\om(u,Tu)=a$ and $u=Tu$. 
\\
\\
If (ii) or (iii) hold, then by Lemma \ref{Godwillmake}, $u$ is a fixed point of $T$. Uniqueness follows from (\ref{help3}) and (\ref{extree}).
\end{proof}
\noindent
If $\varphi:[0,\infty) \times [0,\infty) \to [0,\infty)$ is the product, then $(\varphi_i)$ is satisfied for all $i$ with $k \in [0,\frac{1}{s})$. Since $\kappa=1$ for b-metric spaces, we obtain the following corollary: 
\begin{corollary}
Let $(X,d,s)$ be a complete b-metric space, with $s \geq 1$, and $T:X \to X$ a map such such that for some $k \in [0,\frac{1}{s})$ and for all $x,y \in X$,
\begin{equation*}
d_\om(Tx,Ty) \leq k \max\{d_\om(x,y), d_\om(x,Tx), d_\om(y,Ty), d_\om(x,Ty), d_\om(y,Tx)\}.
\end{equation*}
If $X$ is $T$-orbitally complete, then $T$ has a unique fixed point.
\end{corollary}

\section{Conclusion and Further Work}

The results of this work highlight the analytic structure underlying iterative processes in O-metric spaces. Rather than beginning directly with contractive mappings, we developed a framework for resolving inequalities governed by the binary operation $\om$ and auxiliary control functions. The introduction of generalized $\om$-series provides a flexible mechanism for analyzing patterned compositions arising from the generalized triangle inequality. In particular, the separation principle established in this paper allows inequalities of the form
\[
u \le v \,\om \, \varphi(k,u)
\]
to be resolved without reference to a specific iterative scheme. This analytic viewpoint makes it possible to derive convergence properties of sequences independently of fixed point arguments and subsequently apply these results to obtain \'{C}iri\'{c}-type theorems as consequences of the general theory.
\\

\noindent
A notable feature of the framework is the systematic use of control functions. The function $\varphi$ plays a role analogous to a multiplicative scaling in classical contraction theory, while the control function $\gamma$ acts as a stabilizing operation satisfying
\[
\max\{\gamma(u,v),\gamma(u\om v,a)\}\le \max\{u,v\}.
\]

\noindent
From a broader perspective, the algebraic substitutions that appear naturally in this framework suggest an intriguing interpretation. In classical analysis, iterative arguments rely fundamentally on the operations of addition, multiplication, and averaging. In the present setting these roles are effectively replaced by the operations $\om$, $\varphi$, and $\gamma$ as follows:
\[
+ \, \longrightarrow \, \om,
\qquad
\times \, \longrightarrow \, \varphi,
\qquad
\frac{u+v}{2} \, \longrightarrow \, \gamma(u,v).
\]
This observation hints at the possibility of developing an alternative calculus-like framework in which the basic operations are no longer derived from field structures but instead arise from more general binary operations equipped with suitable monotonicity and control properties. Although such a perspective is only implicit in the present work, it suggests that the analytic mechanisms developed here may have relevance beyond the context of fixed point theory.
\\
\noindent
%
%

\noindent
Another direction concerns the algebraic nature of the range of the metric. In this paper the values of the O-metric are taken in intervals of the real line endowed with a binary operation $\om$. 
This raises the possibility that O-metrics could naturally take their values in more general algebraic systems, such as non-associative structures including loops or related algebraic objects. Exploring the extent to which the analytic framework developed here can be extended to such settings may provide a deeper understanding of metric-type geometries in non-associative environments.
\\

\noindent
Some open problems naturally arise from the present work. One problem  is to determine whether the separation principle for $\om$-inequalities can be extended under weaker monotonicity assumptions on the operation $\om$. Overall, the results presented here suggest that O-metric spaces provide a fertile setting in which analytic, algebraic, and metric ideas interact. Further investigation of these interactions may reveal new structures and techniques for studying nonlinear problems in generalized metric environments.

\section*{Authors' Statements}

\subsection*{Conflict of interest}
The authors state no conflict of interest.

\subsection*{Availability of data and materials} 
Not Applicable

\subsection*{Funding} 
 None

\subsection*{Authors' Contributions} 
All authors worked on the research, read and approved the final manuscript.

\subsection*{Acknowledgments}
 None

\end{document}